\documentclass[...]{amsart}

\usepackage{graphicx}

\newtheorem{theorem}{Theorem}[section]
\newtheorem{lemma}[theorem]{Lemma}

\newtheorem{definition}[theorem]{Definition}
\newtheorem{example}[theorem]{Example}

\newtheorem{proposition}[theorem]{Proposition}

\newtheorem{remark}[theorem]{Remark}

\begin{document}

\title{Typed Topology And Its Application To Data Sets}


\author{Wanjun Hu }
\address{Department of Math, CS And Physics\\ Albany State University\\ 504 College Dr., Albany GA 31705}
\curraddr{}
\email{whu@asurams.edu}
\thanks{}



\dedicatory{To the memory of W.W. Comfort}

\commby{}

\begin{abstract}
 The concept of $typed$ $topology$ is introduced. In a typed topological space, some open sets are assigned "types", and topological concepts such as closure, connectedness can be defined using types. A finite data set in $R^2$ is a typically typed topological space. Clusters calculated by the DBSCAN algorithm for data clustering can be well represent in a finite typed topological space. Other concepts such as tracks, port (starting points),  type-p-connectedness, p-closure-connectedness, indexing, branches are also introduced for finite typed topological spaces. Finally, $left-r$ and $up-left-r$ type open sets are introduced for finite data sets in $R^2$, so that tracks, port, branches can be calculated.    
\end{abstract}

\maketitle

\section{Introduction}
Given a set $X$, a topology $\mathcal T\subseteq P(X)$
 on $X$ is a family of subsets, called $open$ $subsets$, that
 satisfy the conditions: (1) 
$\emptyset\in \mathcal T$; (2) $X\in\mathcal T$; (3) 
If $U,V\in\mathcal T$, then $U\cap V\in\mathcal T$; 
(4) If $\mathcal U\subset\mathcal T$, 
then $\bigcup\{U: U\in\mathcal U\}\in \mathcal T$.
   For any point $x\in X$, the 
neighborhood system of $x$ is the subfamily 
$\mathcal U(x)\subseteq \mathcal T$ such that 
$x\in U$ for any $U\in \mathcal U(x)$.  
    The neighborhood system helps define the concepts of $closure$, $connectedness$ and variety of other concepts in general topology. For more information about general topology, the reader is referred to \cite{engel}.

From 1920s to 1980s, there was huge development in general topology. Many concepts and topological invariants were established. During that period of time, the concept of "neighborhood"  attracted a lot interesting in other areas. For instance, in 1936, the "Topological Psychology"  \cite{lewin} was proposed, in which an individual's behavior is connected to his/her environment, which can be conceived as the "convergence" in some sense. 

In Data Mining, the concept of neighborhood has been popular too. One major question in data mining is to find clusters inside a data set, which can be conceived as data values that are $close$ to each other. In 1951, Fix et al \cite{fix}, studied the K-Nearest Neighbors algorithm (KNN) (see also \cite{cover}). In 1967, J. Macqueen \cite{macqueen} studied the K-mean algorithm for classification and clustering, which is based on how $closer$ a point is to a cluster. See also \cite{k-means1}, \cite{k-means2} and \cite{k-means3}.
   Around that time, clustering problem was investigated \cite{sJohnson}, \cite{Harding}, \cite{hartigan}, \cite{Jardine2}, in Marketing \cite{green},  in Archaeology \cite{Hodsome},  in Biology \cite{k-means3},\cite{Jardine},\cite{jardine3},  in Information System  \cite{short}, and   in Psychology\cite{sJohnson} \cite{Hubert}. In 1996, Ester et al \cite{dbscann} studied the density-based algorithm for discovering clusters and proposed the celebrated Density-Based Spatial Clustering Of Applications With Noise (DBSCAN) algorithm. For a good reference to data clustering, the reader is referred to the ASA-SIAM Series book \cite{dataclusteringBook}.
 
Above mentioned three classical data clustering algorithms utilize the concepts of neighborhood and metric, which are typical general topological concepts. For finite data sets inside an Euclidean space, they are quantitative, and carry several metrics, each of which can benefit particular data clustering algorithms. A list of those metrics can be found in \cite{dataclusteringBook}. 
  However, from topological point of view, those data sets are discrete, where a particular metric does not make any difference. A finite topological space satisfies almost every topological property when it is $T_1$, i.e., every singleton is closed. It satisfies almost nothing when it is not $T_1$. In 1966, R.E. Stong \cite{finitespace}, published an article on finite topological spaces. That article  pretty much put an end to the study of finite topological spaces. 
    
  The situation becomes interesting when we try to ask data mining/machine learning questions within finite topological spaces. For instance, there are problems such as the clustering, classification, recommendation, forecasting, outlier detection, ranking, optimization, and  regression, i.e., those typical data mining questions, inside a finite topological space. One can also ask topological questions within a data set. For instance, how to describe connectedness and closure in a finite data set?
  
  In 2009, Carlson published the article "Topology and data" \cite{carlsson}, which provided a survey of using algebraic topological tools on data sets. A (very large) data set can be represent by a set of nodes and connections, where each node stands for a clustering. Since then, the research area of Topological Data Analysis started blooming. 
  Around that time, the author of this article worked on the so-called "typed topological spaces", which was more set theory and logic originated, and was based on infinite sets. The author was inspired by \cite{carlsson} and turned to study "typed topological spaces" from the perspective of finite data sets.   
  
   In this article, we propose the so-called $typed$ $topology$, in which some open sets are assigned $types$. Topological concepts can be defined by using typed open sets.  For instance, the closure of a set can be re-defined using only open neighborhoods of certain types. It is interesting to see that the closure concept will not be the same anymore. In fact, it will  have several tracks, other than a "thin" layer of boundary. It turns out that the DBSCAN algorithm can be well described by the concept of $p$-closure (here, $p$ is a type). We will also introduce other type based concepts, and use them to analyze finite topological spaces. For instance, we define type based clustering, tracks, port (starting points), horizontal indexing, surgery, and branches in a data set.     

    The arrangement of this article is as follows. In section 2, we define the so-called $typed$ $topological$ $space$. Then, we redefine some of the basic concepts from general topology, such as type closure, type connectedness, type connected component. Some examples are provided. We will show how the DBSCAN algorithm is described by type closures. 
    In section 3, we study the cluster $p\vdash tr(x)$. Various concepts are introduced, such as tracks, port (the starting points of a cluster), straight space, surgery, $p$-closure connected sets, and surgery. We also discuss a method to separate a mixed cluster into disjoint clusters.  
    In section 4, we study a local and horizontal axis defined on $p\vdash Track_i(x)$. In particular, we investigate how to expand the axis of some points to a larger area $q\vdash tr(x)$ defined by a stronger type $p<<q$.   
    In section 5, we study branches inside a cluster $p\vdash tr(x)$. In addition, we introduce the types of $left-r$, $right-r$, $up-left-r$ and $up-right-r$, and use them to equip a data set in $R^2$ into a directed typed topological space. An example will be given with calculation of tracks and branches.

\section{Typed topological spaces}
\indent Given a topological space $(X, \mathcal T)$, and a partially ordered set $<P,\leq>$,  we combined them into a typed topological space. Each element in $P$ is considered as a type. 

\begin{definition}\label{typedSpace}
 Let $(X, \mathcal T)$ be a topological space and let $<P, \leq>$ be a partially ordered set. For every $x\in X$, set $\mathcal U(x)=\{O\in\mathcal T:x\in O\}$. If for every $x\in X$, there exists a partial function
$\sigma_x: \mathcal U(x)\rightarrow P$ satisfying
for every $U,V\in domain(\sigma_x)$, $\sigma_x(U)\leq \sigma_x(V)$ when $U\subseteq V$,
then, the 5-tuple 
$(X, \mathcal T, P, \leq, \{\sigma_x: x\in X\})$ is called a typed topological space.

The function $\sigma_x$ (or in short $\sigma$) is called the $assignment$ $function$. When $\sigma_x(U)= p$, we say $U$ is $of$ $type$ $p$, and write it as 
$p\vdash U(x)$.
$\Box$
\end{definition}

As a simple example, every topological space is automatically a typed topological space, where the partially ordered set $<P, \leq>$ is the topology $\mathcal T$ with $\subseteq$ as the partial order relation, i.e., 
$<P,\leq>=<\mathcal T, \subseteq>$. 
The assignment functions will be $\sigma_x(U)=U$ for every $x\in U\in\mathcal T$.

In general, for any subfamily  $\mathcal T'\subseteq\mathcal T$, let $P=\{\mathcal T'\}$. The set $P$ is equipped with partial order $\leq$ such that $\mathcal T'\leq \mathcal T'$. The assignment function $\sigma_{x}$ is defined as $\sigma_{x}(B) = \mathcal T'$ for $B\in \mathcal U_x\cap\mathcal T'$. A typed topological space is obtained.

The first countable topological space is a typed topological space, where the assignment function $\sigma_x$ maps a countable neighborhood base of the point $x$, $\{U_i: i<\omega\}$ with $U_{i+1}\subseteq U_i$, to the set $P=\{\frac{1}{n}: n=1, 2, ...\}$ with the usual real line order $\leq$. We may call a first countable topological space a type $(P, \leq)$ space.

Further, given a set $X$ and a set of binary relationships $R$ on $X\times X$, for any $x\in X$ and  any $r\in R$, we define $U_r(x) = \{y\in X: r(x, y)\}$. Those $U_r(x)'s$ form a neighborhood basis of $x$. Let $P=\{U_r(x): x\in X, r\in R\}$. The partially order is the subset relation. 
The assignment function 
$\sigma_x$ is defined as $\sigma_x(U_{r}(x)) = U_r(x)$. A typed topological space is obtained.
 
\begin{example}\label{DBSCAN}
Let $X$ be a finite subset on the XY-plane. The set $X$ is equipped with the usual Euclidean topology. For each 
$x\in X$, and any positive real number $p$, the open disk of $x$ with radius $p$ is defined as
$U_p(x)=\{y: distance(x, y)<p\}$ . Let $P$ be a set of some real numbers. The set $P$ comes with the usual linear order $\leq$.   For instance, one may choose a very small positive real number $\epsilon$, and set $P=\{\epsilon, 2\epsilon, 4\epsilon,..., 2n\epsilon\}$. 
The assignment function $\sigma_x$ is defined as $\sigma_x(U_p(x)) = p$ for $p\in P$. A typed topological space can be created accordingly. 
$\Box$
\end{example}

\begin{example}\label{qualitative}
Let $X$ be the set of all high school students in a county. Let $P$ be a set of binary relationships, e.g., "friends", "close friends", "best friends", "class mates", "members on a math club", "members on a basket ball team", etc. There is a natural partial order relation on $P$. For instance, "best friends" are "close friends", and "close friends" are "friends". For $x\in X$, $p\in P$, set $U_p(x)=\{y\in X: p(x, y)\}$, and let $\sigma_x(U_{p}(x)) = p$. A typed topological space is obtained.
$\Box$
\end{example}

Example \ref{qualitative} has a particular meaning. It is a purely qualitative data set.

\begin{remark}
In the Definition \ref{typedSpace}, each $\sigma_x$ is a partial function. It is worth noting that for an open neighborhood $U\in\mathcal U(x)$ such that $\sigma_x(U)=p$  (or $p\vdash U(x)$),  and a point $y\in U$ satisfying $y\neq x$, we may not have $\sigma_y(U)=p$. In other words, an open set $U$ that is a type $p$ open neighborhood of $x$, may not be a type $p$ open neighborhood of a point $y\in U$. $\Box$
\end{remark}

\indent Topological concepts such as closure, connectedness and connected components can be re-defined using types. 
  
\indent In the following, we will use $p\vdash U(x)$  to indicate that $U$ is an open neighborhood of $x$ of type $p$.

\begin{definition}\label{defClosure}
Let $(X, \mathcal T, P, \leq, \{\sigma_x: x\in X\})$ be a typed topological space. 
Given a set $A\subseteq X$ and a type $p\in P$, a point $x\in X$ is called a $p$-$accumulation$ $point$ of $A$, 
if for every $p\vdash U(x)$, we have 
$U\cap A\neq\emptyset$. The set 
$A\cup\{x\in X: x~ is~an~accumulation~point~of~A\}$ 
is called the $direct$ $closure$ of $A$, and is denoted 
 $p\vdash CL_1(A)$.
 $\Box$
\end{definition}

\indent In general, we may not have 
$p\vdash CL_1(p\vdash CL_1(A))=p\vdash CL_1(A)$. Let $x\in (p\vdash CL_1(p\vdash CL_1(A)))$, i.e., $x$ is a $p$-accumulation point of $p\vdash CL_1(A)$. For any  
$p\vdash U(x)$, we have  $U\cap (p\vdash CL_1(A))\neq\emptyset$. However, for any point $y\in U\cap (p\vdash CL_1(A))$, the set $U$ may not be an open neighborhood of $y$ of type $p$, i.e.,$\sigma_y(U)\neq p$, which means that we may not have $U\cap A\neq\emptyset$. Therefore, $x$ may not be a $p$-accumulation point of $A$.

\begin{definition}
We can use the notation
$p\vdash CL_2(A)$ for  $p\vdash CL_1(p\vdash CL_1(A))$, 
$p\vdash CL_3(A)=p\vdash CL_1(p\vdash CL_2(A)))$, ... $p\vdash CL_n(A)=p\vdash CL_1(p\vdash CL_{n-1}(A))$.

 The $transitive$ $p$-$closure$ is defined as 
$p\vdash tr(A) = \bigcup\{p\vdash CL_n(A): n\in N\}$. 

\indent When $A=\{x\}$, we use $p\vdash CL_1(x)$  to denote $p\vdash CL_1(\{x\})$, and $p\vdash tr(x)$ to denote $p\vdash tr(\{x\})$. $\Box$
\end{definition}

\indent We may not have 
$p\vdash tr(p\vdash tr(A)) = p\vdash tr(A)$ either. However, when the assignment function $\sigma$ satisfies an additional condition, that statement becomes true. 

\begin{definition}\label{leastNeighborhood}
Let everything be as in Definition \ref{typedSpace}. Given a type $p\in P$, we call $\sigma$ has least $p$-neighborhood, provided that for any $x\in X$, there exists a neighborhood $p\vdash U(x)$ such that for any neighborhood $p\vdash V(x)$, the statement $U\subseteq V$ holds. The least $p$-neighborhood of $x$ is denoted $p\vdash U_{min}(x)$.
$\Box$
\end{definition}

\begin{proposition}
When $\sigma$ has least $p$-neighborhood for every point in $ X$, a point $x$ is a 
$p$-accumulation point of a subset $A$ if and only if 
$(p\vdash U_{min}(x))\cap A\neq\emptyset$.
$\Box$
\end{proposition}

When $X$ is finite, the assignment function $\sigma$ can be easily amended to has least $p$-neighborhood for any type $p\in P$ and any point $x\in X$.

\begin{proposition}
Let everything be as in Definition \ref{typedSpace}. Given a type $p\in P$, if the assignment function $\sigma$ has least $p$-neighborhood, then for any $A\subseteq X$ we have $p\vdash tr(p\vdash tr(A)) = p\vdash tr(A)$.  

\end{proposition}

\begin{proof} If $x$ is a $p$-accumulation point of $p\vdash tr(A)$, i.e., $x\in (p\vdash CL_1(p\vdash tr(A)))$,
 then 
 $(p\vdash U_{min}(x))\cap (p\vdash tr(A))\neq\emptyset$. 
 Let $y\in (p\vdash U_{min}(x))\cap (p\vdash tr(A))$. 
 Then either 
$y\in A$ or $ y\in (p\vdash cl_n(A))$ for some $n\in N$. 
In the first case, we have $x\in (p\vdash CL_1(A))$.
In the second case, we have $x\in (p\vdash CL_1(p\vdash CL_n(A)) = p\vdash CL_{n+1}(A)$. Hence $x\in (p\vdash tr(A))$. So $p\vdash CL_1(p\vdash tr(A))\subseteq (p\vdash tr(A))$.

Inductively, we can show that 
 $p\vdash CL_n(p\vdash tr(A))\subseteq (p\vdash tr(A))$ for all $n$. Therefore, $p\vdash tr(p\vdash tr(A)) = p\vdash tr(A)$.
\end{proof}

\begin{definition}\label{connectedness}
  Let everything be as in Definition \ref{typedSpace}. A subset $A\subseteq X$ 
is called type-$p$-connected, if there do not exist 
two collections of $p$-neighborhoods 
$\{p\vdash U(x_i): i\in I\}$  and 
$\{p\vdash U(x_j):j\in J\}$ satisfying:
\begin{enumerate}
\item both $I \neq\emptyset$ and $J\neq\emptyset$;
\item $A=\{x_i: i\in I\}\cup\{x_j: j\in J\}$; and 
\item $(\bigcup\{U(x_i): i\in I\})\cap (\bigcup\{U(x_j): j\in J\})=\emptyset$.
\end{enumerate} 
$\Box$
\end{definition}

\begin{example}
On the $XY$-plane, let $X$ be the set 
$\{(0,0),(0,1), (1,1), (1,0)\}$ as in Example \ref{DBSCAN}. 
Then $X$ is type-$p$-connected for any real number $p>1$, 
and type-$p$-disconnected when $0<p\leq 1$.
$\Box$
\end{example}

\begin{remark}
In Definition \ref{connectedness}, we require that the collections satisfying $A=\{x_i: i\in I\}\cup\{x_j: j\in J\}$.
If that requirement is dropped, then the concept to be defined is similar to the traditional definition for connectedness, which is different than what is defined in Definition \ref{connectedness}.

 For instance, let 
$X=\{(0,0), (0,1), (1,1), (1,0), (2, 0)\}$ be as in Example \ref{DBSCAN}. Set $p=\sqrt{2}$.
Then $X$ is type-$p$-connected.
 However, we can find two collections
$\{U_p((2, 0))\}$ and 
$\{U_p((0,1))\}$. The set $U_p((2,0))\cap X=\{(2,0), (1, 0)\}$, 
and $U_p((0,1))\cap X=\{(0,1), (1,1), (0, 0)\}$, which are disjoint and cover $X$.  
$\Box$
\end{remark}

\begin{lemma}\label{closureConnect}
Let everything be as in Definition \ref{connectedness}. 
If $A\subseteq X$ is type-$p$-connected, then
\begin{enumerate}
\item  $p\vdash CL_n(A)$ is type-$p$-connected for $n>0$, and 
\item $p\vdash tr(A)$ is also type-$p$-connected.
\end{enumerate}
\end{lemma}

\begin{proof} For (1) We first show that $p\vdash CL_1(A)$ is type-$p$-connected.
For a contradiction, we assume that there are two collections  of $p$-neighborhoods 
$\{p\vdash U(x_i): i\in I\}$ with $I\neq\emptyset$ and 
$\{p\vdash U(x_j): j\in J\}$ with $J\neq\emptyset$ satisfying: 
(1) $p\vdash CL_1(A)=\{x_i: i\in I\}\cup\{x_j: j\in J\}$; and 
(2) $\bigcup\{U(x_i): i\in I\}\cap \bigcup\{U(x_j): j\in J\}=\emptyset$. Then the two collections 
$\{p\vdash U(x_i): i\in I\}$ and 
$\{p\vdash U(x_j): j\in J\}$ will be two collections that covers $A$. Hence one of them 
must be empty, say $I\cap A=\emptyset$, and 
 $A\subseteq \bigcup\{U(x_j): j\in J\}$.
Since $I\neq\emptyset$, 
there exists an $i\in I$ such that $x_i\in ((p\vdash CL_1(A))\setminus A)$. Hence $x_i$ is a $p$-accumulation point of $A$, and we have $U(x_i)\cap A\neq\emptyset$, which further implies that 
$U(x_i)\cap \bigcup\{U(x_j: j\in J\}\neq\emptyset$. It is a contradiction. Therefore, $p\vdash CL_1(A)$ is also type-$p$-connected.

Inductively, we will have $p\vdash CL_n(A) =p\vdash CL_1(p\vdash CL_{n-1}(A))$ is also type-$p$-connected. 
\\

(2) To show that $p\vdash tr(A)$ is type-$p$-connected, we assume for a contradiction that there exist two collections   of $p$-neighborhoods 
$\{p\vdash U(x_i): i\in I\}$ with $i\neq\emptyset$ and 
$\{p\vdash U(x_j): j\in J\}$ with $J\neq\emptyset$ satisfying: 
(1) $p\vdash tr(A)=\{x_i: i\in I\}\cup\{x_j: j\in J\}$; and 
(2) $\bigcup\{U(x_i): i\in I\}\cap \bigcup\{U(x_j): j\in J\}=\emptyset$. 
For a similar argument as in the proof for (1), we can assume that 
$A\subseteq \bigcup\{U(x_j): j\in J\}$. If $x_i\in p\vdash CL_1(A)\setminus A$, then $x_i$ is a $p$-accumulation point of $A$, which implies $U(x_i)\cap A\neq\emptyset$ and $U(x_i)\cap \bigcup\{U(x_j: j\in J\}\neq\emptyset$. 
Hence $x_i\in \bigcup\{p\vdash U(x_j): j\in J\}$. That shows 
$p\vdash CL_1(A)\subseteq \bigcup\{p\vdash U(x_j): j\in J\}$.   

Inductively, we can 
show that $p\vdash CL_n(A)\subseteq  \bigcup\{p\vdash U(x_j): j\in J\}$, which implies 
$p\vdash tr(A)\subseteq  \bigcup\{p\vdash U(x_j): j\in J\}$. That is a contradiction with our assumption. 
\end{proof}

\begin{definition}\label{type-p-component}
Let everything be as in Definition \ref{connectedness}. For any $x\in X$, the type-$p$-connected component $p\vdash C(x)$ is the union of 
all type-$p$-connected sets that contain $x$.
$\Box$
\end{definition}

\begin{theorem}\label{thm110}
Let everything be as in Definition \ref{type-p-component}. 
   For any $x\in X$ and $p\in P$, we have 
\begin{enumerate}
\item the type-$p$-connected component $C(x)$ is type-$p$-connected; 
\item  $p\vdash CL_1(C(x))= p\vdash C(x)$; and 
\item  $p\vdash tr(C(x))=p\vdash C(x)$.
\end{enumerate}
\end{theorem}
\begin{proof}
For (1), we assume for a contradiction that there exists two collections of $p$-neighborhoods 
$\{p\vdash U(x_i): i\in I\}$ with $I\neq\emptyset$ and 
$\{p\vdash U(x_j): j\in J\}$ with $J\neq\emptyset$ satisfying: 
(1) $p\vdash C(x)=\{x_i: i\in I\}\cup\{x_j: j\in J\}$; and 
(2) $\bigcup\{U(x_i): i\in I\}\cap \bigcup\{U(x_j): j\in J\}=\emptyset$.
Without loss of generality, we assume $x\in \bigcup\{p\vdash U(x_i): i\in I\}$. 
By definition, the $p\vdash C(x)$ is the union of all type-$p$-connected subsets that containing $x$. 
Hence for any $y\in C(x)$, there exists a type-$p$-connected subset $A\subseteq C(x)$ such that 
$\{x, y\}\subseteq A$. Using similar argument as that in the proof of  Lemma \ref{closureConnect}, we 
have $\{x, y\}\subseteq A\subseteq \bigcup\{p\vdash U(x_i): i\in I\}$. 
The arbitrary choice of $y$ implies that $C(x)\subseteq \bigcup\{p\vdash U(x_i): i\in I\}$, which is a contradiction.

\indent (2) and (3) are true according to Definition \ref{connectedness} and the Lemma \ref{closureConnect}.
\end{proof}

\indent Generally speaking, if $B\subseteq A$, then $(p\vdash tr(B))\subseteq (p\vdash tr(A))$. In particular, if $y\in (p\vdash tr(x))$, 
then $(p\vdash tr(y))\subseteq (p\vdash tr(x))$.

 A quantitative data set usually satisfies the following property of being "symmetrically typed". 

\begin{definition}\label{symmtricType}

A typed topological space $(X, \mathcal T, P, \leq, \{\sigma_x: x\in X\})$ is called $p$-symmetrically typed, if for every two 
points $x, y\in X$ satisfying $y\notin (p\vdash U(x))$ for a neighborhood $p\vdash U(x)$ of $x$, there exists a neighborhood $p\vdash V(y)$ such that $x\notin (p\vdash V(y))$.
$\Box$
\end{definition}

The following proposition is straight-forward.
\begin{proposition}
When the space $X$ is $p$-symmetrically typed and $\sigma$ has least $p$-neighborhood, it is easily to check that $x\notin (p\vdash U_{min}(y))$ if and only if  $y\notin (p\vdash U_{min}(x))$. Equivalently, we have $x\in (p\vdash U_{min}(y))$ if and only if  $y\in (p\vdash U_{min}(x))$.
 $\Box$
\end{proposition}

\begin{theorem}\label{ThmSymmetricType}
Let $(X, \mathcal T, P, \leq, \{\sigma_x: x\in X\})$ be a $p$-symmetrically typed topological space for a type $p\in P$. Assume that the assignment function $\sigma$ has least $p$-neighborhood. Then, for any $x\in X$, the following statements are true.
\begin{enumerate}

\item $y\in (p\vdash CL_1(x))$ if and only if $x\in (p\vdash CL_1(y))$;

\item  $p\vdash C(x)=p\vdash tr(x)$;

\item  for any $y\in (p\vdash tr(x))$, $p\vdash tr(x)=p\vdash tr(y)$.

\end{enumerate}
\end{theorem}

\begin{proof} (1) It is directly from the definition of being $p$-symmetrically typed.

(2) By Lemma \ref{closureConnect}, $p\vdash tr(x)$ is type-$p$-connected. So $p\vdash tr(x)\subseteq p\vdash C(x)$.  Assume for a contradiction that $p\vdash tr(x)\neq p\vdash C(x)$. Set $A=(p\vdash C(x))\setminus (p\vdash tr(x))$, and $B=p\vdash tr(x)$. 
  Since $p\vdash tr(p\vdash tr(x))=p\vdash tr(x)$, for every point $y\in A$, 
there exists $p\vdash U(y)$ satisfying $(p\vdash U(y))\cap B=\emptyset$.  Then, $(p\vdash U_{min}(y))\cap B=\emptyset$, and the generality of $y$ implies that $\bigcup\{p\vdash U_{min}(y): y\in A\}\cap B=\emptyset$. 

By Definition \ref{symmtricType}, for each $z\in B$, there exists $p\vdash U(z)$ satisfying 
$y\notin (p\vdash U(z))$.  Since $p$ has least neighborhood, we can assume that both 
$p\vdash U(y) = p\vdash U_{min}(y)$ and $p\vdash U(z)=p\vdash U_{min}(z)$. 
Hence $y\notin \bigcup\{p\vdash U_{min}(z): z\in B\}$. The arbitrary selection of $y\in A$ implies 
that $A\cap \bigcup\{p\vdash U_{min}(z): z\in B\} =\emptyset$.

For $z\in B$, and any $w\in (p\vdash U_{min}(z))$, one has $z\in (p\vdash CL_1(w))$. By (1), we have 
$w\in (p\vdash CL_1(z))$. Hence $(p\vdash U_{min}(z))\subseteq B$. Therefore, 
we have $\bigcup\{p\vdash U_{min}(z): z\in B\} = B$, which implies $\bigcup\{p\vdash U_{min}(y): y\in A\}\cap \bigcup\{p\vdash U_{min}(z): z\in B\} = \emptyset$.

The two collections of type $p$ open neighborhoods $\{p\vdash U_{min}(z): z\in A)\}$ and $\{p\vdash U_{min}(y): y\in B\}$, witness that $p\vdash C(x)$ is not type-$p$-disconnected. We arrive at a contradiction.

(3) For any $y\in (p\vdash tr(x))$, we have $y\in (p\vdash CL_n(x))$ for some number $n>0$, which means that $(p\vdash U_{min}(y))\cap (p\vdash CL_{n-1}(x))\neq\emptyset$. Let $y_{n-1}\in (p\vdash U_{min}(y))\cap (p\vdash CL_{n-1}(x))$. Then by definition, $(p\vdash U_{min}(y_{n-1}))\cap (p\vdash CL_{n-2}(x))\neq\emptyset$. So we can choose $y_{n-2}\in  (p\vdash U_{min}(y_{n-1}))\cap (p\vdash CL_{n-2}(x))$. Inductively, we can define $y=y_n, y_{n-1}, y_{n-2},..., y_1$ satisfying $y_{i+1}\in (p\vdash CL_1(y_i))$ for all $1\leq i<n$. By (1), we have also $y_i\in (p\vdash CL_1(y_{i+1}))$ for all $i$. Hence $x\in (p\vdash tr(y))$, and $(p\vdash tr(x))\subseteq (p\vdash tr(y))$. Therefore $p\vdash tr(x) = p\vdash tr(y)$.
\end{proof}

The algorithm of Density-Based Spatial Clustering Of Applications With Noise (DBSCAN) was studied in 1960s,
and improved later. It was chosen for the test of time award in 2014.  The algorithm choose a distance $\epsilon$ and 
a number $minPoints\geq 3$. Points of a data set are classified 
into (1) core points, (2) directly-reachable points, (3) reachable points, and (4) outliers. 

A point $x$ is called a $core$ $point$ if $|\{y\in X: distance(x,y)<\epsilon\}| \geq  minPoints$. 
  For a core point $x$, points inside $\{y\in X: distance(x,y)<\epsilon\}$ are called $directly$ $reachable$ points from $x$. 
  If a point $z$ is directly reachable from $y$ and $y$ is directly reachable from $x$, then $z$ is called  $reachable$ from  $x$. In general, if there is a path $x_1=x, ..., x_n=z$ such that each $x_{i+1}$ is directly reachable from $x_i$, then $z$ is reachable from $x$. 
  A point is not reachable from any other points is called an $outlier$ or a $noise$ $point$.

Comparing with our definition of $p\vdash tr(x)$, the DBSCAN algorithm calculates exactly $p\vdash tr(x)$. Here, the type $p$ represents the open neighborhood of radius $\epsilon$.   
 A directly reachable point from $x$ is a point $y\in(p\vdash Cl_1(x))$, and  
 a reachable point $y$ satisfies $y\in (p\vdash CL_n(x))$. 
 Finally,  $x$ is a core point if  $|p\vdash CL_1(x)|\geq minPoints$. 

  Hence, the DBSCAN algorithm calculates the $cluster$ as closure in the form of $p\vdash tr(x)$. Further, the concept of density-connectedness defined in DBSCAN algorithm is equivalent to our definition of  type-$p$-connectedness.

\section{Clusters as closure}

\indent We assume that  $X$ is finite for the rest of this article. 

\indent The clusters calculated in DBSCAN algorithm may be in the form of $p\vdash tr(x)$. Several $p\vdash tr(x_i)$'s can be mixed together as in Figure \ref{starightSurgery}(b). 

In this section, we will study some structures of such clusters.
   
\begin{definition}\label{def_border}
Let $(X, \mathcal T, P, \leq, \{\sigma_x: x\in X\})$ be a typed topological space. For any set $A\subseteq X$ and any type $p\in P$, the $Tracks$ of $A$, denoted $p\vdash Track_i(A)$ for $0<i<\infty$, are defined as follows.
\begin{enumerate}
  \item Set $p\vdash Track_0(A) =  A$;
  \item set $p\vdash Track_1(A) = (p\vdash CL_1(A))\setminus A $;
  \item for $i>1$, set $p\vdash Track_i(A) = (p\vdash CL_i(A))\setminus (p\vdash CL_{i-1}(A))$. 
\end{enumerate}
$\Box$
\end{definition}

\indent A quick observation is that when $p\vdash Track_i(A)=\emptyset$ for some $i$, then both $p\vdash CL_j(A)=p\vdash CL_i(A)$ 
and $p\vdash Track_j(A)=\emptyset$ hold true for all $j\geq i$. Further, 
$p\vdash tr(A) = p\vdash CL_i(A)$ is also true. 

For the cluster $p\vdash tr(x)$, we may also write it as $p\vdash Track (x)$. Since $X$ is finite, we use $|p\vdash Track(x)|$ to denote the smallest such integer $i$ that $p\vdash Track_i(A)=\emptyset$.
Certainly, $p\vdash tr(x) = \bigcup\{p\vdash Track_i(x): i<|p\vdash Track(x)\}$.

In the following, we introduce a class of rather simple typed topological spaces. 

\begin{definition}\label{straight}
  Let everything be as in Definition \ref{def_border}. Given a type $p\in P$, for any point $x\in X$,   the cluster $p\vdash tr(x)$ is called  $straight$ if for any $j>i>0$ and any two points $y,z$ satisfying $y\in(p\vdash Track_i(x))$ and $z\in(p\vdash Track_j(x))$, 
  the statement $y\in (p\vdash CL_1(z))$ implies $z\in(p\vdash CL_1(y))$
 (equivalently, $z\in(p\vdash U_{min}(y))$  implies $y\in(p\vdash U_{min}(z))$). 

  The space $X$ is called $p-straight$ if for every $x$, the cluster $p\vdash tr(x)$ is straight. (See Figure \ref{starightSurgery}(a))
  $\Box$
\end{definition}

\begin{figure}
  \includegraphics[width=4in]{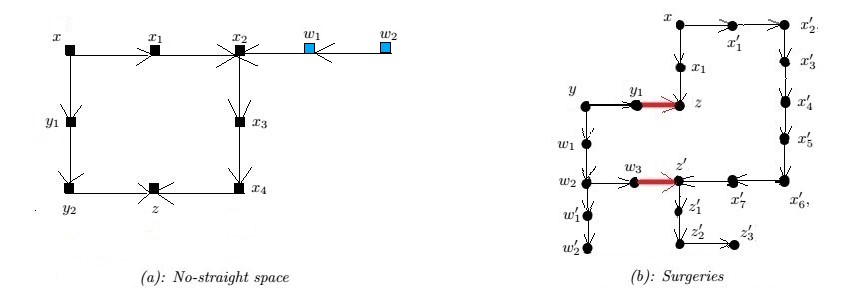}
  \caption{A cluster that is not straight and locations for surgery }
  \label{starightSurgery}
  \end{figure}

\begin{lemma}\label{straight2}
Let everything be as in Definition \ref{straight}. Given a type $p\in P$, assume that $X$ is $p$-straight. For any point $x, y\in X$ satisfying $y\in(p\vdash Track_i(x))$ for some $i$, if $z\in (p\vdash CL_1(y))$, then either $z\in (p\vdash Track_{i+1}(x))$, or $z\in (p\vdash Track_i(x))$ or $z\in(p\vdash Track_{i-1}(x))$. 
\end{lemma}

\begin{proof}
Since $z\in (p\vdash CL_1(y))$, the statement $z\in (p\vdash CL_{i+1}(x))$ is true. Then, we have that $z\in (p\vdash Track_j(x))$ for some $j\leq i+1$. Assume for a contradiction that $z\in (p\vdash Track_j(x))$ for some $j<i-1$. Since $X$ is $p$-straight, the statement $z\in (p\vdash CL_1(y))$ implies $y\in (p\vdash CL_1(z))$. Hence we have $y\in (p\vdash CL_k(x))$ for $k=j +1<(i-1)+1 = i$. It contradicts with the assumption that $y\in(p\vdash Track_i(x))$. Hence, we must have $z\in (p\vdash Track_k(x))$ for some $k\in\{i-1, i, i+1\}$.
\end{proof}

\begin{example}\label{examNonStraight}
  Figure \ref{starightSurgery} (a) shows an example of a non-straight space. The arrow between two points, e.g., from $x$ to $y_1$, indicates the relationship $y_1\in(p\vdash CL_1(x))$. The type $p$ means directed link with length $p=1.01$. The set $p\vdash tr(x)=\{x, y_1, y_2, z, x_1, x_2, x_3, x_4\}$. The two points $w_1, w_2$ are not inside $p\vdash tr(x)$.
  The point $z$ is inside $p\vdash Track_3(x)$ and $x_4\in (p\vdash Track_4(x))$. However, 
   $p\vdash U_{min}(z)=\{y_2,z, x_4\}$ while $p\vdash U_{min}(x_4)=\{x_3, x_4\}$, which means that 
  $z\not\in (p\vdash U_{min}(x_4))$ while $z\in(p\vdash U_{min}(z))$. 
  
  When the arrow between $z$ and $x_4$ is changed to double arrow, the space becomes straight.
  
   Figure \ref{starightSurgery} (a) demonstrates a situation when we try to establish indexing of points on a horizontal axis in next section, a confusion occurs to index the point $z$. $\Box$
\end{example}

Clusters  maybe a mixture of several $p\vdash tr(x_i)$'s. To study those clusters, we identify the so-called starting points (or entrances). The following concept of "port" is introduced.

\begin{lemma}\label{port}
Let everything be as in Definition \ref{def_border}. For any subset $D\subseteq X$ and any type $p\in P$,
there exists a subset of $D$, denoted  
 $port_p(D)$, satisfying
 \begin{enumerate}
 
\item  $D\subseteq p\vdash tr(port_p(D))$; and 
\item  for every two distinct points $z, w\in port_p(D)$, both $z\notin (p\vdash tr(w))$ and $w\notin (p\vdash tr(z))$ hold true. 
 \end{enumerate}
\end{lemma}

\begin{proof}
We define an equivalent relation $\equiv$ on $D$ such that $y\equiv z$ if and only if both $y\in (p\vdash tr(z))$ and 
$z\in (p\vdash tr(y))$ are true, or equivalently $p\vdash tr(y)=p\vdash tr(z)$

On the quotient $D/\equiv$, we define a partial order $\leq_p$ in the way that   
$[y]\leq_p [z]$ if and only if $z\in (p\vdash tr(y))$, where $[y]$ ($[z]$) represents the equivalent class in $D/\equiv$ that contains $y$ (resp., $z$). When $[y]\leq_p [z]$, we have $z\in (p\vdash tr(y))$, and $(p\vdash tr(z))\subseteq (p\vdash tr(y))$. If $y'\in [y]$ and $z'\in [z]$, then $z'\in (p\vdash tr(z')) = p\vdash tr(z)$ and $p\vdash tr(y')= p\vdash tr(y)$. Hence, 
$z'\in (p\vdash tr(z))\subseteq (p\vdash tr(y))=(p\vdash tr(y'))$, i.e., $z'\in (p\vdash tr(y'))$ is true.

 Let $D'/\equiv$ be the set of all minimum equivalent classes in $(D/\equiv, \leq_p)$.  For each equivalent class $[z]$ 
 in $D'/\equiv$, we choose a point $z\in [z]$, and form the set $port_p(D)$. 
 
 It is easy to show that for any point 
 $y\in D$, there exists a point $z\in port_p(D)$, such that $[z]\leq_p [y]$. Hence $y\in (p\vdash tr(z))$ and $D\subseteq (p\vdash tr(port_p(D))$. Therefore, the statement (1) is true.
 
 (2) follows directly from the definition of $port_p(D)$.
\end{proof}

When the topology is symmetrically typed, the set $port_p(D)$ will be a singleton of any point in $D$. It is interesting to see that quantitative data sets are usually symmetrically typed. However, it will not be much helpful to identify the starting points of any cluster in those spaces. 

In section 5, we will provide a way to equip a data set in $R^2$ with the so-called $left-r$, $right-r$, $up-left-r$ and $up-right-r$ types, and obtain a non-symmetrically typed topological space. With those types, the starting points for clusters can be calculated. 

It is worth noting that without the typed topology, one will use all possible open sets to study properties, which leads nowhere. Some properties become visible when we use only certain types of open sets.

In the following, we study how to separate clusters into individual $p\vdash tr(x_i)$'s.

\begin{definition}\label{def_closureConnected}
Let everything be as in Definition \ref{def_border}. Let $D, D'$ be two subsets of $X$, and let $p$ be a type. The sets $D$ and $D'$ are called 
$p$-$closure$ $disjoint$ if 
$(p\vdash tr(D))\cap (p\vdash tr(D'))=\emptyset$.
A subset $D\subseteq X$ is called $p$-$closure$ $connected$ if it cannot be partitioned into two non-empty subsets 
$E, E'$ such that $E$ and $E'$ are p-closure disjoint.    
$\Box$
\end{definition}

\begin{lemma}\label{closureCompCondition}
Let everything be as in Definition \ref{def_closureConnected}.
A set $D\subseteq X$ is $p$-closure connected if and only if for any two points $z, w\in D$, there exists a sequence $z=w_1, w_2, ..., w_t=w$ satisfying: (1) $w_i\in D$; and (2) $(p\vdash tr(w_i))\cap (p\vdash tr(w_{i+1}))\neq\emptyset$ is true for all $i=1, ..., t-1$. 
\end{lemma}

\begin{proof}
($\rightarrow$) Assume that $D$ is $p$-closure connected. For any two points $z, w\in D$, there 
exists $x_1, x_2\in port_p(D)$ satisfying $z\in (p\vdash tr(x_1))$ and $w\in(p\vdash tr(x_2))$. Set $D_1=\{y\in port_p(D):~ there~ exists~ x_1=w_1, w_2, ..., w_t=y~satisfying: (1) w_i\in port_p(D); ~and ~(2)  (p\vdash tr(w_i))\cap(p\vdash tr(w_{i+1}))\neq\emptyset ~for~i<t\}$. From the definition of $D_1$, one can see that, if $z\in port_p(D)\setminus D_1$, then $(p\vdash tr(z))\cap (p\vdash tr(y))=\emptyset$ is true for every $y\in D_1$. Hence $(p\vdash tr(z))\cap (p\vdash tr(D_1))=\emptyset$. The random choice of $z$ implies that $(p\vdash tr(port_p(D)\setminus D_1))\cap (p\vdash tr(D_1))=\emptyset$, which implies $port_p(D)\setminus D_1=\emptyset$ and $x_2\in D_1$. 

($\leftarrow$) Assume for a contradiction that $D$ is not $p$-closure connected. 
Then $D$ can be split into $D=D_1\cup D_2$ such that 
$(p\vdash tr(D_1))\cap(p\vdash tr(D_2))=\emptyset$. Let $z\in D_1$ and $w\in D_2$. There exists
 $z=w_1, w_2, ..., w_t=w$ satisfying : (1) $w_i\in D$; and (2) 
 $(p\vdash tr(w_i))\cap (p\vdash tr(w_{i+1}))\neq\emptyset$ is true for all $i<t$. 
 Set $i_0=max\{i: w_j\in D_1~is~true~for~all~j\leq i\}$. Then 
 $w_{i_0}\in D_1$ and $w_{i_0+1}\in D_2$. By the definition of $D_1, D_2$, we have 
 $(p\vdash tr(w_{i_0}))\cap (p\vdash tr(w_{i_0+1}))=\emptyset$, which is a contradiction.

\end{proof}

  \begin{remark}
A $p$-closure connected set may not be type-$p$ connected. For example, in a straight space $X$, for any two points $x, y$ satisfying $y\in (p\vdash Track_2(x))$, the set $D=\{x, y\}$ is not type-$p$ connected. It is however $p$-closure connected. 
   On the other hand, a type-$p$ connected set is also $p$-closure connected.
$\Box$
\end{remark}

\begin{lemma}\label{p-closure-decom}
Let everything be as in Definition \ref{def_closureConnected}. For any subset $D\subseteq X$ and a type $p\in P$, there is a unique partition $D=D_1\cup D_2\cup ...\cup D_t$ such that each $D_i$ is $p$-closure connected, and the union of any two or more $D_i$'s will be $p$-closure disconnected.
\end{lemma}
\begin{proof}
We start with any point $d_1\in D$. Set $D_1=\{w\in D: \exists w_1=d_1, w_2, ..., w_m=w~satisfying~: (1) w_i\in D; and (2)(p\vdash tr(w_i))\cap (p\vdash tr(w_{i+1}))\neq\emptyset~ is~ true~ for~ all~ i<m-1\}$, as in Lemma \ref{closureCompCondition}. Then, 
we choose a point $d_2\in D\setminus D_1$ and define $D_2$ in the same way as $D_2$. The process continues until all points in $D$ are arranged. 
\end{proof}

\begin{example}\label{figure1b}
  Figure \ref{starightSurgery}(b) shows another cluster $D'$. The arrow, for instance from $x$ to $x_1$, indicates the relationship $x_1\in(p\vdash CL_1(x))$.  
  The type $p$ means  directed link with length $p=1.01$. The two points $x,y$ are the starting points, i.e., $port_p(D') = \{x, y\}$. The set $D'$ 
  is $p$-closure connected. It is a mixture of two clusters $p\vdash tr(x)$ and $p\vdash tr(y)$. 
  
  The point $z$ is inside both $p\vdash Track_2(y)$ and $p\vdash Track_2(x)$. The point $z'$ is in both $p\vdash Track_5(y)$ and $p\vdash Track_8(x)$. It is not a non-straight situation as defined in Definition \ref{straight}, since $x\neq y$. However, the point 
  $z'$ also posts a problem when we try to indexing the point starting from both $x$ and $y$.  
  $\Box$
\end{example}

To solve the problem created by the point $z'$ in above Example \ref{figure1b}, we will investigate a new method, called "surgery". Use this new method, we can separate the cluster $D'$ in Example \ref{figure1b} into disjoint $p\vdash tr(x)$ and $p\vdash tr(y)$.

\begin{definition}\label{surgery}
Let $(X, \mathcal T, P,\leq, \{\sigma_x: x\in X\})$ be a typed topological space. For any two distinct points $y\neq z\in X$ and any type $p\in P$,
\begin{enumerate}
\item The $p$-$cut$ at $(z,y)$, denoted $Cut(z, y)$, is the modification of the least neighborhood 
$p\vdash U_{min}(z)$ to $p\vdash_{Cut(z,y)} U_{min} (z):=(p\vdash U_{min}(z))\setminus \{y\}$. Hence, $z\notin (p\vdash_{Cut(z, y)} CL_1 (y))$.

\item When both $z\notin (p\vdash tr(y))$ and $y\notin (p\vdash tr(z))$ hold true, the $p$-$surgery$ on $(z,y)$, denoted $S(z,y)$ (or simply $S$), is 
the process to modify
 the least neighborhood $p\vdash U_{min}(w)$ for each $w\in (p\vdash tr(y))\cap (p\vdash tr(z))$, 
 to $p\vdash_S U_{min}(w) := (p\vdash U_{min}(w))\setminus [(p\vdash tr(z))\setminus (p\vdash tr(y))]$.

\end{enumerate}
 For any point $y\in X$, the notation $p\vdash_{S} tr(y)$ denotes the cluster $p\vdash tr(y)$ after the surgery $S$.
$\Box$
\end{definition}

 \begin{example}
 In Figure \ref{starightSurgery}(a), the two points $(z, x_4)$ can have a cut. After the $Cut(z, x_4)$, the arrow from $x_4$ to $z$ will be removed, i.e., $z\notin (p\vdash_{Cut(z, x_4)} CL_1 (x_4)$.

 In Figure \ref{starightSurgery}(b), the two points $x, y$ satisfy the conditions for a surgery. One can calculate that $(p\vdash tr(x))\cap(p\vdash tr(y))=\{z, z', z'_1, z'_2, z'_3\}$. The surgery $S$ at $(y,z)$ consists of two cuts at $(y_1, z)$ and $(w_3, z')$. Before the surgery $S$, we have $p\vdash tr(y) = \{y, y_1, z, w_1, w_2, w'_1, w'_2, w_3, z', z'_1, z'_2, z'_3\}$, and $p\vdash tr(x)=\{x, x_1, z, x'_1, x'_2, ..., x'_7, z',$ $z'_1, z'_2, z'_3\}$
 After the surgery $S$, we have $p\vdash_S tr(y) = \{y, y_1, w_1, w_2, w_3, w'_1, w'_2\}$, while $p\vdash_S tr(x)=p\vdash tr(x)$.
$\Box$
\end{example}

The following proposition is obvious.

\begin{proposition}\label{afterSurgery}
Let everything be as in Definition \ref{surgery}. Let $S$ be the $p$-surgery on $(z, y)$. 
Then, (1) $p\vdash_S tr(y)=p\vdash tr(y)$; (2) $(p\vdash_S tr(z))\cap (p\vdash_S tr(y))=\emptyset$; and (3)
$(p\vdash_S tr(y))\cup(p\vdash_S tr(z)) = (p\vdash tr(y))\cup(p\vdash tr(z))$.
$\Box$
\end{proposition}
 
 \begin{theorem}\label{cutForStraight}
 Let $(X,\mathcal T, P, \leq, \{\sigma_x: x\in X\})$ be a typed topological space. For any $x\in X$ and $p\in P$, there is sequence of cuts (as in Definition \ref{surgery}), denoted $SC$,  so that in the resulting space, $p\vdash_{SC} tr(x)$ is straight.
 \end{theorem}
 
 \begin{proof}
For any $i<|p\vdash Track (x)|$, and any  $y\in (p\vdash Track_i(x))$. Set 
$Check(y) = \{z\in (p\vdash U_{min} (y))\setminus (p\vdash CL_i(x)): y\notin (p\vdash U_{min} (z))\}$.
If $Check(y)\neq\emptyset$, then set $cut(y)=\{(y, z): z\in Check(y)\}$. Let $SC=\bigcup\{cut(y): y\in (p\vdash tr(x))\}$. The set $SC$ can be arranged in any order. We perform the cut at each pair inside $SC$. The set $p\vdash_{SC} tr(x)$ will be straight in the resulting space.
 \end{proof}

  A cluster $D$ like  what is shown in Figure \ref{starightSurgery}(b) has starting points $port_p(D)=\{x, y\}$. If we choose $x$ as a reference point for $D$, then the point $y$ surrounds $x$, i.e., 
  $(p\vdash tr(y))\cap(p\vdash tr(x))\neq\emptyset$. In general,  there may be points, e.g., 
  $z$ surrounding $y$, i.e., 
  $(p\vdash tr(y))\cap(p\vdash tr(z))\neq\emptyset$ while $(p\vdash tr(x))\cap(p\vdash tr(z))=\emptyset$, and points surrounding $z$ and so on. 
 
 We want to perform a sequence of surgeries $PS$ so that, after those surgeries,  $(p\vdash_{PS} (y))\cap (p\vdash_{PS}(z))=\emptyset$ is true for every two distinct points $y, z\in port_p(D)$. 
 
 Another consideration is that when doing surgery on $(y, x)$ with $y$ surrounding $x$, we want to keep the set $p\vdash tr(x)$ intact.  After the surgeries, the set $p\vdash tr(x)$ will not change, while links from $(p\vdash tr(x))\cap(p\vdash tr(y))$ to $p\vdash tr(y)$ are removed, which means that $p\vdash_{PS} tr(y)=(p\vdash tr(y))\setminus [(p\vdash tr(y))\cap (p\vdash tr(x))]$. Then, we remove the common part of $p\vdash_{PS} tr(y)$ and $p\vdash tr(z)$ from $p\vdash tr(z)$ if $z$ surrounding $y$, and so on.  
  
  We first introduce the following sequence of surgeries, called "separation surgeries".
  
\begin{definition}\label{separationSurgery}
Let $(X, \mathcal T, P,\leq, \{\sigma_x: x\in X\})$ be a typed topological space. Given a type $p\in P$,  for any ordered sequence of points $<y_1, y_2, ..., y_n>$. The $p$-$separation$ $surgeries$ $SS$ on that sequence is the sequence of $p$-surgeries  $<(y_2, y_1),$ $(y_3, y_1), ..., (y_n, y_1), (y_3, y_2)$ $,..., (y_n, y_2), ..., (y_4, y_3),$ $...,$ $(y_n, y_3)$ $,..., (y_{n-1}, y_n)>$. 

Further, when $y_i=y_j$ for some $i<j$, the $p$-surgery $(y_j, y_i)$ is the removal of $y_j$ from the original sequence of points $<y_1, y_2, ..., y_{j-1},$ 
 $y_{j+1},$ $...,$  $y_n>$.
$\Box$
\end{definition}

Our definition of "$p$-surgery" in Definition \ref{surgery} will not be performed on such pairs of points as $(y,y)$, or $(z, y)$ when $z\in (p\vdash tr(y))$, since in both cases, the first point in the pair will disappear. However, in above definition of $p$-separation surgeries, we do allow repetitions of a point $y$ in an ordered sequence. In that case, we will consider a $p$-surgery,  such as $(y, y)$,  
a removal of other occurrences of  that point $y$ from the ordered sequence. 

The following proposition is obvious.
\begin{proposition}\label{SS}
Let everything be as in Definition \ref{separationSurgery}.
After the $p$-separation surgeries $SS$, the property $(p\vdash_{SS} tr(z))\cap(p\vdash_{SS} tr(w))=\emptyset$ holds  for all distinct $z, w\in\{y_1, y_2, ..., y_n\}$. 
$\Box$
\end{proposition}  
  
The surrounding relationship among points inside $port_p(D)$ can be represent in a general tree structure.   
  
\begin{definition}\label{surroundingD0}
Let $(X, \mathcal T, P, \leq, \{\sigma_x:x\in X\})$ be a typed topological space. Let $p$ be a type. For any $p$-closure connected set $D$, let $d_0\in D$ be a point. The set $D$ can be positioned on a general tree structure, called $Surrounding$ $Tree$ and denoted $ST(D)$, which satisfies the following conditions:
\begin{enumerate}
\item The point $d_0$ is at the root of the tree.  Set $L_0=\{d_0\}$.
\item The set of all points at level $i$ is denoted $L_i$.
\item For each node $d\in L_i$, the set of its children, denoted $Children(d)$, is the set $\{e\in D\setminus (L_0\cup L_1\cup...\cup L_{i-1}): (p\vdash tr(e))\cap (p\vdash tr(d))\neq\emptyset\}$.
\item Set $L_{i+1}=\bigcup\{Children(d): d\in L_i\}$.
\item The $p$-surgeries at each node $d$ is $PS_i(d)=\{(e, d): e\in Children(d)\}$ .
The $p$-surgeries   $PS_i$ at level $i$  is the sequence of $p$-surgeries of $\{PS_i(d): d\in L_i\}$ plus the $p$-separation surgeries on $L_{i+1}$.
\end{enumerate}
$\Box$
\end{definition}  
  
  For two points $d, e$ such that $e\in Children (d)$ and $d\in L_i$, it may happen that $e\in Children (d')$ for another point $d'\in L_i$. Hence the point $e$ may repeats inside $L_{i+1}$. After the $p$-surgeries $PS_i(d)$ and $PS_i(d')$, we will conduct the $p$-separation surgery on $L_{i+1}$. The point $e$ will be eliminated from one of $d$ and $d'$. 
  For any point $d'\in L_i$, if $e\notin Children (d')$, by definition of $Children(d')$, we have $(p\vdash tr(e))\cap(p\vdash tr(d'))=\emptyset$.
  
  Further, for two points $d\in L_i$ and $d'\in L_j$ with $j>i$, we have $Children(d)\cap  Children (d')$ $=\emptyset$. 
  
  This proves (3) of the following lemma. Item (1) and (2) are obvious, since the surgeries do not lose any points.  
  
  \begin{lemma}\label{procedure}
  Let everything be as in Definition \ref{surroundingD0}. After the sequence of $p$-surgeries $PS$ of $<PS_0$, $PS_1$, ..., $PS_t>$, where $t$ is the height of the surrounding tree $ST(D)$, the following statements are true:
  \begin{enumerate}
  \item   each point $d\in D$ is placed at one and only one node in the tree;
  \item $p\vdash tr(D)=p\vdash_{PS} tr(D)$;
  \item $(p\vdash_{PS} tr(d))\cap (p\vdash_{PS} tr(d'))=\emptyset$ is true for every two distinct points $d,d'\in D$. 
  \end{enumerate}
  $\Box$ 
  \end{lemma}

\section{Local Horizontal Axis And Extension}

For any $x\in X$ and $p\in P$, we can define a local horizontal axis originated at $x$. The axis is defined on the set $p\vdash tr(x)$ (or $p\vdash Track (x)$).

\begin{definition}\label{axisTickMark}
Let $(X, \mathcal T, P, \leq, \{\sigma_x: x\in X\})$ be a typed topological space.
 We use the notation $index_{p, x}(y)$ to denote the assignment of a point $y\in (p\vdash tr(x))$ to a tick mark on a local horizontal axis. 
\begin{enumerate}
 \item  $index_{p, x}(x) = 0$;
 \item if $y\in(p\vdash Track_j(x))$, then $index_{p, x}(y)=j$. 
\end{enumerate}
$\Box$
\end{definition}

To extend the indexing to points outside $p\vdash tr(x)$, we will need another type $q$. 
In the following, we introduce the concept of "stronger type" and also a class of rather simple spaces, called "uniformly typed" space.
It turns out the indexing can be expanded smoothly to $q\vdash tr(x)$ for those spaces for $p<<q$.

\begin{definition} \label{uniformlyTyped}
Let everything be as in Definition \ref{axisTickMark}. Let $p<q $ be two types.
\begin{enumerate}
\item The type $q$ is called a $stronger$ $type$ than $p$, written as $p<< q$, if for all $x\in X$ and any $U, V\in domain(\sigma_x)$, $\sigma_x(U)=p$ and $\sigma_x(V)=q$ implies $U\subseteq V$;

\item  The space $X$ is called $(p, q)$-uniformly typed, if for any $j<|p\vdash Track(x)|$ and $k<|q\vdash Track(x)|$, the condition
$(p\vdash CL_j(x))\cap (q\vdash CL_k(x))\neq\emptyset$ implies $(p\vdash CL_j(x))\subseteq (q\vdash CL_k(x))$. 
$\Box$
\end{enumerate}

\end{definition}

When $p<<q$, one has $(p\vdash tr(x))\subseteq (q\vdash tr(x))$. 
For the set $q\vdash tr(x)$, points inside the subset $p\vdash tr(x)$ are well indexed. Due to the unpredictable "shape" of type $q$ neighborhoods, the set $(q\vdash tr(x))\setminus (p\vdash tr(x))$ need additional treatment. 

\begin{remark}\label{nonUniform}
There are typed topological spaces that are not $(p,q)$-uniformly typed. Let $X=\{(0,0), (1,0), (2, 0), (3,0), (4, 0), (1,1)(1,2), (0,2)\}$ be as in Example \ref{DBSCAN}. Let $p$ denote the type of open sets with radius $p=1.01$, and $q$ denote the type of open sets with radius $q=3.01$. Then, for $x=(0,0)$, the point $(0,2)$ is inside both $p\vdash CL_4(x)$ and $q\vdash CL_1(x)$, while the point $(4,0)$ is inside both $p\vdash CL_4(x)$ and $q\vdash CL_2(x)$. 

Data sets from the XY-plane can be equipped with "shaped" type open neighborhoods, which will provide the "direction" of $p$-closures, and reduce the chance for being non-uniformly typed. We will discuss this more in Section 5. 
$\Box$
\end{remark}

\begin{figure}
  \includegraphics[width=4in]{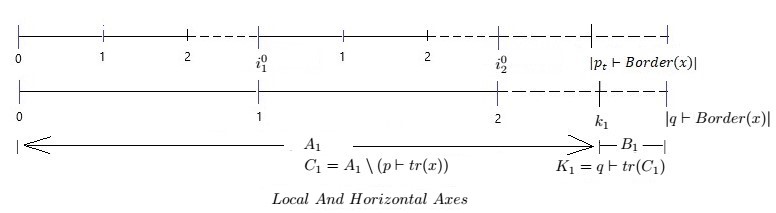}
  \caption{Combined indexing on uniformly typed spaces}
  \label{linearReg1}
  \end{figure}

For $(p,q)$-uniformly typed spaces, we define some notations. See also Figure \ref{linearReg1}.

\begin{definition}\label{pqIndex}
Let everything be as in Definition \ref{uniformlyTyped}. Assume that $X$ is $(p,q)$-uniformly typed, and $p<<q$. 
\begin{enumerate}
\item Set $k_1=min\{k: p\vdash tr(x)\subseteq (q\vdash CL_k(x))\}$;
\item set $i^1_1    =max\{i: (p\vdash CL_i(x))\subseteq (q\vdash CL_1(x))\}$, and inductively set
          $i^1_{j}=max\{i: (p\vdash CL_i(x))\subseteq (q\vdash CL_{j}(x))\}$ for $j=2,3,..., k_1$ 
          			and $i^1_{j}\leq |p\vdash Track (x)|$.
\item set $A_1=q\vdash CL_{k_1}(x)$;
\item set $B_1=(q\vdash tr(x))\setminus A_1$;
\item set $C_1=A_1\setminus (p\vdash tr(x))$;
\item set $K_1=p\vdash tr(C_1)$.
\end{enumerate}
$\Box$
\end{definition}

Above definitions are well-defined for $(p,q)$-uniformly typed spaces. 
   The set $q\vdash tr(x)$ is split into $A_1\cup B_1$.
The set $A_1$ is the first part of $q\vdash tr(x)$ that contains $p\vdash tr(x)$.

   The set $B_1$, on the combined horizontal axes, is after $p\vdash tr(x)$ (See figure \ref{linearReg1}). For $B_1$, the combined axes will start from a new "reference" point $x_2$ that acts like $x$, and the indexing will be extended to  $p\vdash tr(x_2)$. After that, we will need reference points $x_3, x_4, ...$, so that the indexing is extended to $p\vdash tr(x_3), p\vdash tr(x_4),...$. 

   The set $C_1$ can be considered as "parallel" clusters to $p\vdash tr(x)$ inside $q\vdash tr(x)$ that come along with $p\vdash tr(x)$. Those clusters may extend beyond $A_1$ and form $K_1$, and clusters outside $A_1$ may extend to $C_1$.

\begin{lemma}\label{combinedIndex}
Let everything be as in Definition \ref{pqIndex}, where $p<<q\in P$.
The two axes $index_{p, x}(y)$ and $index_{q, x}(y)$ can be combined into an $index_{p,q,x}(y)$ on $p\vdash tr(x)$.
\end{lemma}

\begin{proof}
  For the axis $index_{q, x}(y)$, we can partition it into intervals $[0,1], [1, 2]$, etc (see Figure \ref{linearReg1}). By definition, 
  $index_{q, x}(y)=1$ for all $y\in(q\vdash Track_1(x))$. 
  We can divide
  the interval $[0, 1]$ into $i^1_1$-many sub intervals of equal length, and for $y\in (p\vdash Track_i(x))$ with $i\leq i^1_1$, the combined index $index_{p,q,x}(y)$ is defined as $\frac{i}{i_1^1}$. Similarly, we can divide the interval [1, 2] into $i^1_2-i^1_1$ many sub intervals
  of equal length, and for $i^1_1<i\leq i^1_2$, we index $y\in(p\vdash Track_i(x))$ to $index_{p,q,x}(y)=1\frac{(i-i^1_1}{i_2^1-i_1^1})$. The process continues till all $y\in (p\vdash tr(x))$ have been defined in the combined index $index_{p,q,x}(y)$.
\end{proof}

For any point $y\in (p\vdash Track_k(x))$ for some positive integer $k$, we can calculate its combined indexing in the following way. The result is written as $(k/q).(k\%q)$. The number $k/q$ is calculated as the unique positive integer $t$
satisfying $i^1_{t}\leq k<i^1_{t+1}$. The number $k\%q$ is defined as
$k\%q=k-i^1_{k/q}$. 

 Conversely, for a point $y\in(p\vdash tr(x))$ with a combined indexing $s.t$, we can calculate the number $k$ so that $y\in(p\vdash Track_k(x))$.
  
To index points inside $A_1$, we first calculate the $port_p(A_1)$ as in Lemma \ref{port}. This $port_p(A_1)$ is further split into several $p$-closure connected sets as in Lemma \ref{p-closure-decom}. One of them, say $D$, contains a point $x'$ satisfying $x\in (p\vdash tr(x'))$. The point $x'$ is used as a reference point for $D$. For any other $p$-closure connected set, we also choose a point as a reference point. 

The indexing of a $p$-closure connected set $D$ with a reference point $d_0$ can be defined according to the surrounding tree $ST(D)$.

\begin{definition}\label{joint}
Let everything be as in Definition \ref{pqIndex}. Given a $p$-closure connected set $D\subseteq A_1$ with a reference point $d_0\in D$, assume that $index(p,q, x)(d_0)$ is defined. Let $PS$ be the sequence of surgeries $<PS_0, PS_1, ..., PS_t>$ defined in Definition \ref{surroundingD0}. The indexing of a point inside $p\vdash_{PS} tr(D)$ is defined as follows.

\begin{enumerate}
\item For each $d\in L_1$ (as in Definition \ref{surroundingD0}), 

set $n_{1}=min\{k: (p\vdash tr(d))\cap(p\vdash Track_k(d_0))\neq\emptyset\}$, and 

set $n_{2}=min\{k: (p\vdash Track_k(d))\cap(p\vdash Track_{n_1}(d_0))\neq\emptyset\}$. 

\noindent Then,

set $index_{p,q,x}(d)=(index(p,q, x)(d_0) + n_{1})-n_{2}$.

\noindent For every $e\in (p\vdash_{PS} Track_j (d))$ for some $j$, 

   set $index_{p,q,x}(e) = index_{p,q,x}(d)+j$.

\item Inductively, we assume that each point $d\in L_i$ (as in Definition \ref{surroundingD0}) is properly indexed after surgeries $PS$.
    For each child $e\in Children_{PS}(d)$, where $Children_{PS}(d)$ is the set of children nodes of the tree after surgeries $PS$, 
 
 set $m_{1}=min\{k: (p\vdash tr(e))\cap(p\vdash Track_k(d))\neq\emptyset\}$,  and 
 
 set $m_{2}=min\{k: (p\vdash Track_k(e))\cap(p\vdash Track_{m_1}(d))\neq\emptyset\}$. 

\noindent Then,

 set $index_{p,q,x}(e)=(index_{p,q,x}(d) + m_{1})-m_{2}$.

\noindent  For every $w\in (p\vdash_{PS} Track_j (e))$ for some $j$,

 set $index_{p,q,x}(w) = index_{p,q,x}(e)+j$.
 
\end{enumerate}
$\Box$
\end{definition}

According to Lemma \ref{procedure}, each point $d\in D$ is placed at one and only one node in the surrounding tree $ST(D)$. Hence the indexing defined in above Definition \ref{joint} is well-defined.

\begin{theorem}\label{A-indexing}
Let everything be as in Definition \ref{pqIndex}, where $p<<q\in P$. For a point $x\in P$ and types $p<q\in P$, there exists a sequence of $p$-surgeries $PS$ such that the combined indexing after surgeries $PS$, $index_{p,q,x}(y)$,  can be extended to $A_1\cup K_1$.  
\end{theorem}  

\begin{proof}
For the set $A_1$, by Lemma \ref{port}, there is a $port_p(A_1)$ such that 
(1) $A_1\subseteq p\vdash tr(port_p(A_1))$; and (2) 
  for every two distinct $z, w\in port_p(A_1)$, both $z\not\in (p\vdash tr(w))$ and $w\not\in (p\vdash tr(z))$ are true.

  By Lemma \ref{p-closure-decom}, this $port_p(A_1)$ can be partitioned into $A_1^1\cup A_1^2\cup...\cup A_1^t$ such that each $A_1^i$ is $p$-closure connected and the union of more than one  $A_1^i$'s will be $p$-closure disconnected.   
  
  Without loss of generality, we assume that inside $A_1^1$, there is a point $x'\in A_1^1$ satisfying $x\in (p\vdash tr(x'))$. The point $x'$ will be a reference point for $A_1^1$. According to Definition \ref{joint}, we have a sequence of surgeries $PS$ such that $p\vdash tr(x')$ stay intact, and all points inside $p\vdash tr(A_1^1)$ can be properly indexed after
  $x'$ is properly indexed. 
  
  Let $k$ be such that $x\in (p\vdash Track_k(x'))$. We set $index_{p,q,x}(x')=-k$. Hence points inside $A_1^1$ are indexed according to Definition \ref{joint}.

  For any other $A_1^i (i>1)$, 
   we can consider them as being "parallel with" $p\vdash tr(x)$. For that reason, we can choose a reference point $d_0$ and index it as $0$, i.e., $index_{p,q,x}(d_0)=0$, which is the same as $index_{p,q,x}(x)=0$. Points in $p\vdash tr(A_1^i)$ can be indexed according to Definition \ref{joint}. 
\end{proof}

Points inside $B_1$ may be much more complicate. For instance, there may be points $z\in B_1$ such that $(p\vdash tr(z))\cap port_p(A_1)\neq\emptyset$ or $(p\vdash tr(z))\cap (p\vdash tr(port_p(A_1)))\neq\emptyset$. We introduce another class of typed topological spaces.

\begin{figure}
  \includegraphics[width=4in]{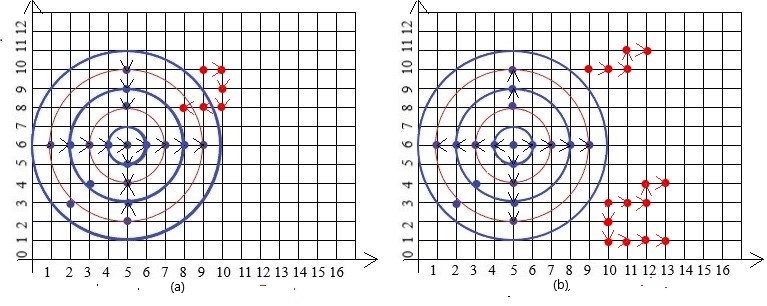}
  \caption{(p,q)-straight examples}
  \label{qpstraight}
  \end{figure}

\begin{definition}\label{pq-straight}
Let $(X, \mathcal T, P, \leq, \{\sigma_x: x\in X\})$ be a typed topological space. For any $x\in X$ and two types $p<q\in P$, the cluster $q\vdash tr(x)$ is called $(p, q)$-straight if for every three points $y, z, w\in (q\vdash tr(x))$ satisfying $z,w\in(p\vdash tr(y))$ the statement $index_{p, y}(z)\leq index_{p,y }(w)$ implies $index_{q,y}(z)\leq index_{q, y}(w)$.
$\Box$
\end{definition} 
\begin{example}
Figure \ref{qpstraight} shows an example of non $(p,q)$-straight space in (a),
and an example of $(p,q)$-straight space in (b). The point $x$ is $x=(5,6)$, and the two types $p,q$ are $p=1.01$ and $q=2.01$.
In both spaces, an arrow from one point $x_1$ to another point $x_2$ indicate the relationship $x_2\in (p\vdash CL_1(x_1))$.

In (a), 

$p\vdash tr(x)=\{(5,6), (6,6), (7,6), (8,6), (9,6), (5,5), (5,4)\}$. 

$k_1=2$, 

$A_1=\{(1,6),(2,3), (2,6), (3,4), (3, 6), (4, 6), (5,2), (5,3), (5,4), (5,5), (5,6), (5,8),$
 $(5, 9), (5, 10), (6,6), (7,6), (8, 6), (8,8), (9, 6) \}$. 

$port_p(A_1)=\{(1,6), (3,4), (5,2), (5, 10), (8, 8)\}$. 

$B_1=\{(2,3), (9, 8), (9, 10), (10, 8), (10, 9), (10, 10)\}$. 

$C_1=\{(1,6),(2,6), (3,4), (3, 6), (4, 6), (5,2), (5,3), (5,4), (5,8),
(5, 9), (5, 10),$ $(8,8)\}$.

"Parallel lines" inside $C_1$ are $\{\{(3,4)\}, \{(5,2), (5,3), (5,4)\}$, 
$\{(5,10), (5, 9),$ $(5,8)\}$  and $\{(8,8)\}$.

In (b), 

$p\vdash tr(x)=\{(1,6), (2,6), (3, 6), (4, 6), (5,2), (5,3), (5,4), (5,5), (5,6), (6,6), (7,6),$
 $(8, 6),(9, 6) \}$. 

$k_1=2$.

$A_1=\{(1,6),(2,6), (3,4), (3, 6), (4, 6), (5,2), (5,3), (5,4), (5,5), (5,6), (5,8),$
 $(5, 9),$  $(5, 10), (6,6), (7,6), (8, 6),  (9, 6) \}$.

$port_p(A_1) = \{(3,4), (5,6), (5, 8)\}$.

$B_1=\{(2,3), (9, 10), (10, 10), (11, 10), (11, 11), (12, 11), 
(13,1), (12,1), (11, 1),
(10,$  $ 1), (10,$  $2), (10, 3), (11, 3),
(12, 3), (12, 4), (13, 4)\}$

$C_1=\{(3,4), (5,8), (5, 9), (5, 10)\}$.

"Parallel lines" inside $C_1$ are $\{(3,4)\},$ and  
$\{(5,8), (5, 9),$ $(5,10)\}$.
$\Box$
\end{example}

The following lemma is obvious. It is needed to extend the index to $B_1$. 

\begin{lemma}
Let everything be as in Definition \ref{pq-straight}. Given a point $x\in X$ and two types $p<<q\in P$, assume that 
the cluster $X$ is $(p, q)$-uniformly typed, and the cluster $q\vdash tr(x)$ is $(p,q)$-straight. For any $y\in B_1$, we have $p\vdash tr(y)\subseteq B_1$.  
$\Box$
\end{lemma}

The set $B_1$ looks like $q\vdash tr(x)$, except for no starting point $x$. For that reason,  
we first identify a "reference point" $x_2$, which serves in the same role as $x$ for $A_1$. Then, we define $k_2, A_2, B_2, C_2, K_2$, just like defining $k_1, A_1, B_2, C_1, K_1$.

\begin{definition}\label{DefReferencePoint}
Let everything be as in Definition \ref{pqIndex}. Set $D=q\vdash Track_{k_1+1}(x)$.
A point $x_2\in B_1$ is called a reference point, if it satisfies   
(1) $x_2\in port_p(D)$; and (2)
$|p\vdash Track (x_2)| = max\{|p\vdash Track (y)|: y\in D\}$.
$\Box$
\end{definition}

\begin{example}
In Figure \ref{qpstraight} (b), the set $D=q\vdash Track_{k_1+1}(x)=\{(9,10), (10,3)\}$. Hence $port_p(D)=\{(9,10), (10,3)\}$. 
Since $|p\vdash tr((9,10))|=4$ and $|p\vdash tr((10,3))| = 5$, we set $x_2=(10,3)$.

It is worth noting that if the set $D=q\vdash Track_{k_1+1}(x)=\emptyset$, then $q\vdash tr(x) = q\vdash CL_{k_1}(x)$ and $B_1=\emptyset$. In the case of Figure \ref{qpstraight} (b), if we shift all red points to the right one unit, then 
$D=q\vdash Track_{k_1+1}(x)=\emptyset$.
$\Box$
\end{example}

After choosing the reference point $x_2$ for $B_1$, 
the set $p\vdash tr(x_2)$ serves as the extension of $p\vdash tr(x)$ inside $q\vdash tr(x)$. 

We define a few notations, which is similar to Definition \ref{pqIndex}.

\begin{definition}\label{indexingNotation}
Let everything be as in Definition \ref{pqIndex}.
\begin{enumerate}
\item Set $k_2=min\{k: k>k_1, and~ (p\vdash tr(x_2))\subseteq (q\vdash CL_k(x))\}$;
\item set $i^2_1    =max\{i: (p\vdash CL_i(x_2))\subseteq (q\vdash CL_{k_1+1}(x))\}$, and inductively set
          $i^2_{j+1}=max\{i: (p\vdash CL_i(x_2))\subseteq (q\vdash CL_{k_1+j+1}(x))\}$ for $j=2,3,...,k_2$ and $i^2_{j+1}\leq |p\vdash Track (x_2)|$;

\item set $A_2=(q\vdash CL_{k_2}(x))\setminus (A_1\cup K_1)$;
\item set $B_2= (q\vdash tr(x))\setminus (A_1\cup A_2\cup K_1)$;
\item set $C_2= A_2\setminus (p\vdash tr(x_2))$; and
\item set $K_2= p\vdash tr(C_2)$.
\item In addition, 

set $r_2=min\{k: \exists y\in (p\vdash Track_k(x))~such~ that~ x_2\in(q\vdash CL_1(y))\}$. 
\end{enumerate} 
$\Box$
\end{definition}

\begin{theorem}\label{localHorizontalAxis}
Let everything be as in Definition \ref{indexingNotation}. Assume that the cluster $q\vdash tr(x)$ is $(p,q)$-straight. There is a sequence $S$ of $p$-surgeries on the typed topological space $X$, such that after the surgeries, the indexing $index_{p,q,x}(y)$ can be extended to $q\vdash tr(x)$. 

\end{theorem}
\begin{proof}
We first decide the index of the reference point $x_2$. 
 By definition of $r_2$, we consider the point $x_2$ is away from $p\vdash Track_{r_2}(x)$ by a distance of the length of $q\vdash CL_1(y)$ for some $y\in (p\vdash Track_{r_2}(x))$. 
 Approximately, the length of the set $q\vdash CL_1(y)$ is about $i^1_{k_1}$, since 
$x_2\in (q\vdash Track_{k_1+1}(x))$, and 
 $i^1_{k_1}$ is the length of that segment of $p\vdash tr(x)$ inside $q\vdash Track_{k_1}(x)$. For that reason, we set
  $$
  index_{p,q,x}(x_2)=r_2+i^1_{k_1}.
  $$

For points inside the set $A_2$, we repeat the process described in Theorem \ref{A-indexing}, which will index $A_2\cup K_2$ properly. 

The process continues to choose a reference point $x_3\in B_2$, index $x_3$ just as indexing $x_2$, and repeat Theorem \ref{A-indexing} again. It will stop until all points inside $q\vdash tr(x)$ are indexed.
\end{proof}

\section{Branches}

 In this section, we define "branches" inside $p\vdash tr(x)$. Then, we define the types of $left-r$, $right-r$ and equip data sets in $R^2$ with those types of open neighborhoods.  We will use an example to demonstrate how to calculate those terminologies defined in previous sections.

\begin{definition}\label{verticalLine}
Let $(X, \mathcal T, P, \leq, \{\sigma_x: x\in X\})$ be a typed topological space. Let $x\in X$ and $p\in P$. 
Each $p\vdash Track_i(x)$ can be partitioned into disjoint type-$p$-connected subsets, i.e., $p\vdash Track_i(x) = C_i^1\cup C_i^2\cup...\cup C_i^{t_i}$, where $C_i^j$ is type-$p$-connected, and the union of two or more $C_i^j$'s is not type-p-connected.
$\Box$
\end{definition}

For the set $p\vdash tr(x)$ and any subset $A\subseteq (p\vdash tr(x))$,
    set $A_i=(p\vdash Track_i(x))\cap A$.

\begin{definition}\label{branch}
Let $X$ be a typed topological space. For any $x\in X$ and $p\in P$,   
a subset $A\subseteq (p\vdash tr(x))$ is called a $branch$ of $p\vdash tr(x)$, if 
  (1) for each $i\leq |p\vdash Track(x)|$,  either $A_i=\emptyset$ or $A_i=C_i^j$  holds true for some $j$; 
  (2) if $A_i=\emptyset$ then $A_j=\emptyset$ is true for all $j>i$;  and
  (3) when both $A_i$ and $A_{i+1}$ are not empty, the statement $A_{i+1}\cap (p\vdash CL_1(A_i))\neq\emptyset$ is also true.  
\end{definition}

In previous sections, we mentioned a method to create $new$ types of open sets for 
a finite data set in $R^2$ with Euclidean distance $d(x,y)=\sqrt{(x_1-x_2)^2+(y_1-y_2)^2}$, where $x=(x_1, y_1), y=(x_2, y_2)$. The purpose for doing that is to obtain non-symmetrically typed spaces from $R^2$. 
  Several concepts, such as tracks, branches, $port_p(D)$, surrounding tree, can be calculated for data sets in $R^2$.

\begin{definition}
Let $X$ be a finite data set $R^2$. 
 For any point  
 $x = (a, b)\in X$, the left-r neighborhood, denoted $L(r, x)$, and the right-r neighborhood, denoted $R(r,x)$, are defined as follows.
 \begin{enumerate}
 \item  $L(r,x)=\{y=(c,d)\in X: d(x, y) \leq r ~and~ c\leq a\}$; and
 \item $R(r,x)=\{y=(c,d)\in X: d(x, y) \leq r ~and~ c\geq a\}$.

 \end{enumerate}
$\Box$
\end{definition}
 
 We include the boundary in the definition of the half neighborhoods, because we are dealing with finite sets. The following theorem shows that above definitions of $L(r,x)$ and $R(r,x)$ are justified.

\begin{theorem}\label{boundaryOpen} 
 Let  $X=\{(x_1, y_1), (x_2, y_2),..., (x_m, y_m)\}$ be a data set in $R^2$. For any $r>0$,  there exists a positive real number $\epsilon$, such that the following holds for all $x\in X$.
\begin{enumerate}
 \item  $L(r,x)= \{y=(c,d)\in X: d(x, y) < (r+\epsilon) ~and~ c< (a+\epsilon)\}$; 
 \item $R(r,x) = \{y=(c,d)\in X: d(x, y) < (r+\epsilon) ~and~ c> (a-\epsilon)\}$. 
\end{enumerate} 

\end{theorem}
\begin{proof}
For each $x=(a, b)\in X$, 

set $U(x)=\{y\in X: d(x,y)>r\}$, and

set $V(x)=\{y=(c,d)\in X: a<c\}$.

\noindent Further, 

set $r_1=min \{d(x,y)-r: ~x\in X, y\in U(x)\}$, and 

set $r_2=min \{c-a: a<c,  x\in X, y\in V(x)\}$.

\noindent Finally, set $\epsilon=min\{r_1, r_2\}$. 
\\

If  $y=(c,d)\in X$ satisfying $d(x,y)<r+\epsilon$ and 
$c<a+\epsilon$, then either (i) $d(x,y)\leq r$ or (ii) $r< d(x,y)< r+\epsilon$. If (ii) is true, then by the definition of $U(x)$, we have $y\in U(x)$ and $d(x, y)-r\geq r_1\geq \epsilon$. Hence $d(x,y)\geq r+\epsilon$, which is a contradiction with (ii). Therefore, we must have (i) $d(x, y)\leq r$. 

Similarly, when $c<a+\epsilon$, we have either (iii) $c\leq a$, or (iv) $a<c<a+\epsilon$. If (iv) holds true, then by the definition of $V(x)$, we have $y\in V(x)$, which means that $c-a\geq r_2\geq\epsilon$. Hence $c\geq a+\epsilon$, which contradicts with (iv). Therefore, we must have (iii) $c\leq a$.
   
Therefore, $y\in L(r,x)$.

Using the same argument, we can find a $\epsilon'$ satisfying item (2). Choosing the smallest of $\epsilon$ and $\epsilon'$ and use that as the requested $\epsilon$ in the statement of the theorem.

\end{proof}

\begin{theorem}\label{th5-5}
Let $X=\{(x_1,y_1), (x_2, y_2), ..., (x_m, y_m)\}$ be a finite data set in $R^2$, and set $\gamma=min\{d(x,y): x\neq y\in X\}$.  Let $P$ be the set of 
$\{left-r, right-r: r = \gamma, 2\gamma, ...,  2n* \gamma \}$. Each point $x\in X$ is equipped with left-r neighborhoods $L(r,x)$'s and right-r neighborhoods $R(r,x)$'s, where $r\in P$. If $2n*\gamma\geq max\{d(x,y): x\neq y\in X\}$, the resulting space is symmetrically typed only when $X$ is on a vertical line, i.e.,  for some constant $a$, $x_i=a$ is true for all $i\leq m$.
\end{theorem}

\begin{proof}
It is easy to see that $X$ equipped with the left-r and right-r open neighborhoods is a typed topological space. The space $X$ is not symmetrically typed. 

Let $p$ be a type of  left-$r\in P$. If $y\in L(r, x)$ for two distinct points $x=(a,b)$ and $y=(c,d)$, then $d(x,y)\leq r$ and $c\leq a$. By definition of $p$, we have $p\vdash U_{min}(x) = L(r, x)$ and $x\in(p\vdash CL_1(y))$. 

For the same $x$ and $y$, the statement $x\in L(r, y)$ is true only when $c=a$, i.e., they are on the same vertical line.
   When $2n*\gamma\geq max\{d(x,y): x\neq y\in X\}$, that conclusion applies to any pair of points $x,y$. Therefore, the space is symmetrically typed only when $X$ is on a same vertical line. 

In addition, for any point $x=(a,b)$ and $p=left-r$, the set $p\vdash CL_1(x)$ is the right half disk with the center at $x$, i.e., $p\vdash CL_1(x) = \{y=(c,d)\in X: d(x,y)\leq r, a\leq c\}$. See
Figure \ref{2DDatasetNonSymTyped} for an example.
\end{proof}

\begin{remark}
For a data set $X\subseteq R^2$, the topology using $left-r$ and $right-r$ types of open neighborhood on $X$ is not the same as the subspace topology from the usual topology on $R^2$. One can consider such a topology as an addition to the subspace topology. $\Box$
\end{remark}

\begin{figure}
  \includegraphics[width=4in]{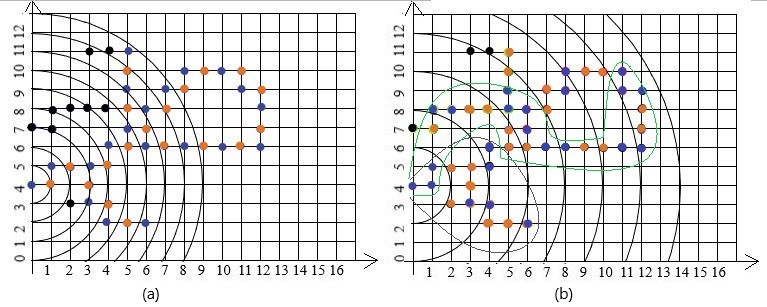}
  \caption{2D data set with non symmetrically typed topology}
  \label{2DDatasetNonSymTyped}
  \end{figure}

In the following, we will use an example to demonstrate how to calculate, for a given point $x$, all tracks, type-$p$-connected decomposition of each track, and all possible branches.

\begin{table}[ht]
 \noindent \caption{Tracks of $p\vdash Track_i(x)$ and $q\vdash Track_i(x)$}
\begin{tabular} {|l |p{5.5cm} |p{5.5cm}|}
\hline
i & $p\vdash Track_i(x)=C_i^1\cup C_i^2\cup ...$ & $q\vdash Track_i(x)=C_i^1\cup C_i^2\cup ...$\\
\hline
1 & $\{(1,4)\}$                   & $\{(1,4), (1,5)\}$\\  \hline
2 & $\{(1,5)\}$					  & $\{(1,7)\}\cup\{(2,3),(2,5),(3,4),$ $(3,5)\}$ \\	\hline
3 & $\{(2,5)\}$					 &  $\{(1,8),(2,8)\}\cup\{(3,3),(4,3),(4,5),(4,6)\}$ \\	\hline
4 & $\{(3,5)\}$					&   $\{(3,8),(4,8),(5,6),(5,7),(6,6)\}\cup\{(4,2),(5,2)\}$ \\	\hline
5 & $\{(3,4)\}\cup\{(4,5)\}$	&   $\{(5,8),(5,9),(6,8),(6,7),$ $(7,6), (8,6)\}\cup\{(6,2)\}$ \\	\hline
6 & $\{(3,3)\}\cup\{(4,6)\}$	& 	$\{(5,10),(5,11)\}\cup\{(7,8),(7,9)\}\cup\{(9,6), (10,6)\}$ \\	\hline
7 & $\{(4,3)\}\cup\{(5,6)\}$	&	$\{(8,9),(8,10)\}\cup\{(11,6),(12,6)\}$ \\	\hline
8 & $\{(4,2)\}\cup\{(5,7)\}\cup\{(6,6)\}$				&  $\{(9,10),(10,10)\}\cup\{(12,7),(12,8)\}$ \\	\hline
9 & $\{(5,2)\}\cup\{(5,8)\}\cup\{(6,7)\}\cup\{(7,6)\}$	&  $\{(11,9),(11,10),(12,9)\}$ \\	\hline
10 & $\{(6,2)\}\cup\{(5,9)\}\cup\{(6,8)\}\cup\{(8,6)\}$	&   \\	\hline
11 & $\{(5,10)\}\cup\{(7,8)\}\cup\{(9,6)\}$		&		\\	\hline
12 & $\{(5,11)\}\cup\{(7,9)\}\cup\{(10,6)\}$		&		\\	\hline
13 & $\{(8,9)\}\cup\{(11,6)\}$						&	\\	\hline
14 & $\{(8,10)\}\cup\{(12,6)\}$						&	\\	\hline
15 & $\{(9,10)\}\cup\{(12,7)\}$						&	\\	\hline
16 & $\{(10,10)\}\cup\{(12,8)\}$					&		\\	\hline
17 & $\{(11,10)\}\cup\{(12,9)\}$						&	\\	\hline
18 & $\{(11,9)\}$								&		\\	\hline	
\end{tabular}
\label{table1}
\end{table}

\begin{example}
Using Figure \ref{2DDatasetNonSymTyped}, the $X=\{(0,4),(0,7), (1,4), (1,5),(1,7),(1,8),$
$(2,3),$ $(2,5),(2,8),(3,3),(3,4),(3,5),(3,8),$ 
$(3,11),$
$(4,2),$
$(4,3),(4,5),(4,6),(4,8),$
$(4,11),$ $(5,2),(5,6),(5,7),(5,8),(5,9),(5,10),(5,11),$
$(6,2),(6,6),(6,7),(6,8),(7,6),$
$(7,8),$ $(7,9),(8,6),(8,9),(8,10),(9,6),(9,10),(10,6),$
$(10,10),$
$(11,6),(11,9),(11,10),$
$(12,6),$ $(12,7),(12,8),(12,9)\}$.

Figure \ref{2DDatasetNonSymTyped}(a) demonstrates the data set $X$ with ideal tracks $p\vdash Track_i(x)$ for the point $x=(0,4)$ and $p=left-1$. The actual tracks are displayed alternately in either red and blue. Points in dark color are not inside $p\vdash tr(x)$. 
   For instance, 
   
   $p\vdash Track_1 (x)=\{(1,4)\}$, 
   
   
   
   
   $p\vdash Track_5 (x)=\{(3,4), (4,5)\}$,   etc.
    
    The points $(2,3),(0,7), (1,7), (1,8), (2,8), (3,8), (3,11), (4,8), (4,11) $ are not inside $p\vdash tr(x)$.
 
Figure \ref{2DDatasetNonSymTyped}(b) demonstrates the data set $X$ with ideal tracks 
of $q\vdash Track_i(x)$, where $q=left-2$. The actual tracks are displayed alternately in either red and blue. Points in dark color are not inside $q\vdash tr(x)$.
   For instance, 
   
   $q\vdash Track_1 (x)=\{(1,4), (1,5)\}$, 
 
   
   

    $q\vdash Track_5 (x)=\{(5,8), (5,9), (6,2), (6,7), (6,8), (7,6), (8,6)\}$,   etc. 
    
    The points $(0,7), (3,11), (4,11) $ are not inside $q\vdash tr(x)$.
$\Box$
\end{example}

Table \ref{table1} shows each $p\vdash Track_i(x)$ and $q\vdash Track_i(x)$ for $x=(0,4)$, and their partitions into $C_i^j$'s.

The following lists all branches in Figure 2(a) with starting point $x=(0,4)$ and type $p=left-1$.
It is in the format of $\{C_1^{j_1}, C_2^{j_2}, ...., C_{k}^{j_{k}}\}$ with $k\leq 18$.
\begin{enumerate}

\item $\{\{(0,4)\}, \{(1,4)\}, \{(1,5)\}, \{(2,5)\}, \{(3,5)\}, \{(3,4)\}, \{(3,3)\}, \{(4,3\}), \{(4,2)\},$ $ \{(5,2)\}, \{(6,2)\}\}$.

\item $\{\{(0,4)\}, \{(1,4)\}, \{(1,5)\}, \{(2,5)\}, \{(3,5)\}, \{(4,5)\}, \{(4,6)\}, \{(5,6\}), \{(5,7)\},$ $ \{(5,8)\}, \{(5,9)\}, \{(5,10\}, \{(5,11)\}\}$.

\item 
$\{\{(0,4)\}, \{(1,4)\}, \{(1,5)\}, \{(2,5)\}, \{(3,5)\}, \{(4,5)\}, \{(4,6)\}, \{(5,6\}), \{(5,7)\},$ $ \{(5,8)\}, \{(6,8)\}, \{(7,8\}, \{(7,9)\}, \{(8,9)\}, \{(8,10)\}, \{(9,10)\}, \{(10,$ $10)\},$   $\{(11,$  $10)\},$  $\{(11,9)\}\}$.

\item
$\{\{(0,4)\}, \{(1,4)\}, \{(1,5)\}, \{(2,5)\}, \{(3,5)\}, \{(4,5)\}, \{(4,6)\}, \{(5,6\}), \{(5,7)\},$ $ \{(6,7)\}, \{(6,8)\}, \{(7,8\}, \{(7,9)\}, \{(8,9)\}, \{(8,10)\}, \{(9,10)\}, \{(10,$ $10)\},$   $\{(11,$  $10)\},$  $\{(11,9)\}\}$.

\item 
$\{\{(0,4)\}, \{(1,4)\}, \{(1,5)\}, \{(2,5)\}, \{(3,5)\}, \{(4,5)\}, \{(4,6)\}, \{(5,6\}), \{(6,6)\},$ $ \{(6,7)\}, \{(6,8)\}, \{(7,8\}, \{(7,9)\}, \{(8,9)\}, \{(8,10)\}, \{(9,10)\}, \{(10,$ $10)\},$   $\{(11,$  $10)\},$  $\{(11,9)\}\}$.

\item 
$\{\{(0,4)\}, \{(1,4)\}, \{(1,5)\}, \{(2,5)\}, \{(3,5)\}, \{(4,5)\}, \{(4,6)\}, \{(5,6\}), \{(6,6)\},$ $ \{(7,6)\}, \{(8,6)\}, \{(9,6\}, \{(10,6)\}, \{(11,6)\}, \{(12,6)\}, \{(12,7)\}, \{(12,$ $8)\},$   $\{(12,$  $9)\}$.
\end{enumerate}

The following lists branches in Figure 2(b) with starting point $x=(0,4)$ and type $q=left-2$.
It is in the format of $\{C_1^{j_1}, \underline{C_2^{j_2}}, C_3^{j_3}, \underline{C_4^{j_4}},...., C_{k}^{j_{k}}\}$ with $k\leq 9$. 

\begin{enumerate}

\item $\{\{(0,4)\}, 
\underline{	\{(1,4), (1, 5)\} }, $
$	\{(1,7)\},
\underline{	\{(1,8), (2, 8)\} },$
$    \{(3,8),(4,8), (5,7), (5,6),$ $(6,6)\},$
$ \underline{   \{(5,9), (5, 8), (6, 8), (6,7), (7,6), (8,6) \} },$
$    \{(9,6), (10,6)\},$
$ \underline{   \{(11,6), (12, 6)\} },$
$    \{(12,7),$ $(12, 8)\}, $
$  \underline{(11,9), (11,10),   \{(12,9)\} }$
$    \}$. See the top enclosed area in Figure 2(b).

\item $\{\{(0,4)\}, 
	\underline{\{(1,4), (1, 5)\}}, $
$	\{(1,7)\},
	\underline{\{(1,8), (2, 8)\}}, $
$    \{(3,8),(4,8), (5,7), (5,6), (6,6)\},$
$   \underline{ \{(5,9), (5, 8), (6, 8), (6,7), (7,6), (8,6) \}},$
$    \{(7,9), (7,8)\},$
$   \underline{ \{(8,10), (8, 9)\}},$
$    \{(9,10),$ $(10, 10)\}, $
$  \underline{  \{(11,9), (11,10), (12, 9)\}}
    \}$.

\item $\{\{(0,4)\}, 
\underline{\{(1,4), (1, 5)\}}, 
	\{(1,7)\},$
$\underline{	\{(1,8), (2, 8)\}},$ 
$   \{(3,8),(4,8), (5,7), (5,6), (6,6)\},$
$\underline{    \{(5,9), (5, 8), (6, 8), (6,7), (7,6), (8,6) \}},$
$    \{(5,11), (5,10)\}
    \}$.

\item $\{\{(0,4)\},
	 \underline{\{(1,4), (1, 5)\}},$
$	  \{(2,3), (2,5), (3,4), (3,5)\},$ 
 $    \underline{\{(3,3), (4,6), (4,5), (4,3)\}},$
$      \{(4,2),(5,2)\}, 
      \underline{\{(6,2)\}}\}$. See the lower enclosed area in Figure 2(b) 
   
\item $\{\{(0,4)\},
	\underline{ \{(1,4), (1, 5)\} }, $
$	\{(2,3), (2,5), (3,4), (3,5)\}, $
$    \underline{\{(4,6), (4,5), (4,3),(3,3)\} }, $
$    \{(3,8),(4,8), (5,7), (5,6), (6,6)\},$
$   \underline{ \{(5,9), (5, 8), (6, 8), (6,7), (7,6), (8,6) \} },$
$    \{(9,6),$  $(10,6)\},$
$    \underline{\{(11,6), (12, 6)\} },$
$    \{(12,7), (12, 8)\}, 
   \underline{ \{(11,9), (11, 10),(12,9)\}}
    \}$.

\item $\{\{(0,4)\}, 
	\underline{\{(1,4), (1, 5)\} }, $
$	\{(2,3), (2,5), (3,4), (3,5)\}, $
$  \underline{ \{(4,6), (4,5), (4,3),(3,3)\} }, $
$   \{(3,8),(4,8), (5,7), (5,6), (6,6)\},$
$  \underline{  \{(5,9), (5, 8), (6, 8), (6,7), (7,6), (8,6) \} },$
$    \{(7,9),$ $(7,8)\},$
$  \underline{  \{(8,10), (8, 9)\} },$
$    \{(9,10), (10, 10)\}, $
$ \underline{   \{(11,10), (11, 9), (12,9)\} }
    \}$.

\item $\{\{(0,4)\}, 
\underline{	\{(1,4), (1, 5)\} }, $
$	\{(2,3), (2,5), (3,4), (3,5)\}, $
$ \underline{   \{(4,6), (4,5), (4,3),(3,3)\} },$ 
$    \{(3,8),(4,8), (5,7), (5,6), (6,6)\},$
$ \underline{   \{(5,9), (5, 8), (6, 8), (6,7), (7,6), (8,6) \} },$
$    \{(5,11),$ $(5,10)\},
    \}$.

\end{enumerate}

The left-r and right-r types provide a direction of clustering from left to right. The space $X$ is straight.

Using Table \ref{table1}, one can verify that $(p\vdash Track_5(x))\cap (q\vdash Track_2(x))\neq\emptyset$ and $(p\vdash Track_5(x))\cap (q\vdash Track_3(x))\neq\emptyset$. Hence $X$ is 
not $(p,q)$-uniformly typed. The other two violations are (1) $p\vdash Track_7(x)$ with $q\vdash Track_3(x)$ and $q\vdash Track_4(x)$, and (2) $p\vdash Track_9(x)$ with $q\vdash Track_4(x)$ and  $q\vdash Track_5(x)$. 

To create a $(p,q)$-uniformly typed space, we can use the enhanced types $up and left-r$, and $up-right-r$, i.e., for a point $x=(a, b)\in X$,

set $UL(r,x) = \{y=(c,d)\in X: d(x,y)\leq r, c\leq a, d\geq b\}$, and

set $UR(r,x) = \{y=(c,d)\in X: d(x,y)\leq r, c\geq a, d\geq b\}$.

When using the $up-left-r$ or $up-right-r$ types, above mentioned violations will be eliminated. It is easy to show that $X$ with the new types will be $(p, q)$-uniformly typed for $p=up-left-1$ and $q=up-left-2$.

{\bf Acknowledgment} The author would like to thank the referee for very valuable suggestions and corrections, which help improve this article a lot.

\bibliographystyle{amsplain}

\end{document}